\documentclass[preprint,times,3p,sort&compress]{elsarticle}

\usepackage{epsfig}
\usepackage[ansinew]{inputenc}
\usepackage{amsmath,amsfonts,latexsym,amssymb,amsbsy}
\usepackage{graphicx}
\usepackage{psfrag} %

\usepackage{float}
\usepackage{color,ulem}
\usepackage{amsmath,amssymb}

\newtheorem{theorem}{Theorem}[section]

\newtheorem{proof}[theorem]{Proof}

\newcommand{\degree}{\,^{\circ}}

\newcommand{\R}{\mathbb{R}}

\newcommand{\beq}{\begin{equation}}
\newcommand{\eeq}{\end{equation}}

\begin{document}
\begin{frontmatter}

\journal{}
\title{Structure preserving integrators for solving linear quadratic
optimal control problems with applications to describe the flight of a quadrotor}

\author{Philipp Bader} %
\ead{phiba@imm.upv.es}
\author{Sergio Blanes} %
\ead{serblaza@imm.upv.es}
\author{{Enrique Ponsoda}%
} \ead{eponsoda@imm.upv.es}
\address{
Instituto de Matem\'{a}tica Multidisciplinar, Building 8G, second
floor,\\ Universitat Polit\`{e}cnica de
Val\`{e}ncia. 46022 Valencia, Spain.}

\begin{abstract}

 We present structure preserving integrators for
solving linear quadratic optimal control problems.
This problem requires the numerical integration of matrix Riccati
differential equations whose exact solution is a symmetric
positive definite time-dependent matrix which controls the
stability of the equation for the state. This property is not
preserved, in general, by the numerical methods. We propose second
order exponential methods based on the Magnus series expansion
which unconditionally preserve positivity for this problem and
analyze higher order Magnus integrators. This method can also be
used for the integration of nonlinear problems if they are
previously linearized. 
The performance of the algorithms is illustrated with the stabilization of a quadrotor which is an unmanned aerial vehicle.
\end{abstract}

\begin{keyword}
Optimal control \sep linear quadratic methods \sep matrix Riccati differential equation \sep second order exponential integrators
\MSC 49J15 \sep 49N10 \sep 34A26
\end{keyword}
\end{frontmatter}

\section{Introduction}

Nonlinear control problems have attracted the interest of researchers in different fields, e.g., the control of airplanes, helicopters, satellites, etc. \cite{BouabdallahTesis,budiyono07otc,castillo} during the last years.
While the extensively studied linear quadratic optimal control (LQ) problems can be used for solving simplified models, most realistic problems are inherently nonlinear.
Furthermore, nonlinear control theory can improve the performance of the controller and enable tracking of aggressive trajectories \cite{castillo05soa}.

Solving nonlinear optimal control problems requires the numerical
integration of systems of coupled non-au\-ton\-o\-mous and nonlinear
equations with boundary conditions for which it is of great interest
to have simple, fast, accurate and reliable numerical algorithms
for real time integrations.

It is usual to solve the nonlinear problems by linearization,
which leads to problems that are solvable by linear quadratic methods.%
In general, they %
require the integration of matrix Riccati
differential equations (RDEs) iteratively.%
The algebraic structure of the RDEs appearing in this problem implies that their solutions are symmetric positive definite matrices,
a property that plays an important role for the qualitative and quantitative solutions of both the control and the state vector.

Geometric numerical integrators are numerical algorithms which
preserve most of the qualitative properties of the exact solution.
However, the mentioned positivity of the solution in this problem
is a qualitative property which is not unconditionally preserved
by most methods, geometric integrators included.

We show that some low order exponential integrators
unconditionally preserve this property, and higher order methods
preserve it under mild constraints on the time step. We refer to
these methods as structure preserving integrators, and they will
allow the use of relatively large time steps while showing a high
performance for stiff problems or problems which strongly vary
along the time evolution.

The aforementioned nonlinearities in the control problems
can be dealt with in different ways.
We consider three techniques 
to linearize the equations and the linear equations are then
numerically solved using some exponential integrators which
preserve the relevant properties of the solution. Since the
nonlinear problems are solved by linearization, we first examine
the linear problem in detail.

The paper is organized as follows: The linear case is studied in Section~\ref{linear}, where we emphasize on the algebraic structure of the equations and the qualitative properties of the solutions.
We next consider some exponential integrators and analyze the preservation of the qualitative properties of the
solution by the proposed methods.
In Section~\ref{nonlinear}, it is shown how the full nonlinear problem can - after linearization - be treated as a particular case of the non-autonomous linear one.
The work concludes with the application of the numerical algorithm to a particular example (control of a quadrotor) in Section~\ref{quadrotor}, with which the accuracy of the exponential methods is demonstrated. Numerical results and conclusions are included.

\section{Linear quadratic (LQ) methods in optimal control problems}
\label{linear}
Let us consider the general LQ optimal control problem
\begin{subequations}\label{eq:lqproblem}\begin{align}
    \label{eq:cost}
    \min_{u\in L^2} \int_{0}^{t_{f}}
                \left( X^{T}(t)Q(t)X(t)+ u^{T}(t)R(t)u(t) \right) \, dt\\
    \label{eq:control1}
    \text{subject to}\quad \dot X(t) = A(t)X(t) + B(t)u(t), \quad X(0)=X_{0},
\end{align}
\end{subequations}
where $\dot X(t)$ is the time-derivative of the state vector
$X(t)\in\R^n$, $u(t)\in\R^m$ is the control, $R(t)\in \R^{m\times
m}$ is symmetric and positive definite, $Q(t)\in\R^{n\times n}$ is symmetric
positive semi-definite, $A\in\R^{n\times n}, B\in\R^{n\times m}$
and $M^{T}$ denotes the transpose of a matrix $M$.

Problems of the type (\ref{eq:lqproblem}) are frequent in many
areas, such as game theory, quantum mechanics, economy, environment
problems, etc., see \cite{jcam,engwerda}, or in engineering models
\cite[ch. 5]{libro}.

The optimal control problem \eqref{eq:lqproblem} is solved,
assuming some controllability conditions, by the linear feedback
controller \cite{Kirk04}
\begin{equation}\label{eq:linear:opt_control}
    u(t) = - K(t) X(t),
\end{equation}
with the gain matrix
\begin{displaymath}
    K(t) \, = \, R^{-1}(t) B^{T}(t)  P(t)\, ,
 \end{displaymath}
and $P(t)$ verifying the matrix RDE
\begin{equation}\label{eq:RDE}
 \dot{P}(t)=-P(t)A(t) -A^{T}(t)P(t) +P(t)B(t)R^{-1}(t)B^{T}(t)P(t) -Q(t),
\end{equation}
with the final condition $P(t_f)=0$. The solution $P(t)$ after
backward time integration of \eqref{eq:RDE} is a symmetric and
positive definite matrix when both $Q(t)$ and $R(t)$ are symmetric
positive definite \cite{abou03mre} (similar results also apply for
the weaker condition $Q(t)$ positive semidefinite under general
conditions on the matrices which make the system
stabilizable and detectable). To compute the optimal control
$u(t)$, we solve for $P(t)$, and plugging the control law into
\eqref{eq:control1} yields a linear equation for the state vector
 \begin{displaymath}
     \dot{X} (t)= \left( A(t)- B(t)R^{-1}(t)B^{T}(t)P(t) \right) X(t),\quad X(0)=X_{0},
 \end{displaymath}
to be integrated forward in time, with which the control is readily computed.
Notice that $ S(t)=B(t)R^{-1}(t)B^{T}(t)$ is a positive
semi-definite symmetric matrix (positive definite if $ \ rank \, B
= n $) and $P(t)$ is a positive definite matrix, hence its product
is also a positive semi-definite matrix, and this is very
important for the stability of the solution for the state vector
and ultimately for the control. A numerical integrator which does
not preserve the positivity of $P(t)$ can become unstable when
solving the state vector.

In this paper, exponential integrators, which belong to the class of Lie group methods (see \cite{applnum,IMNZ} and references therein), are proposed in order to solve the RDE \eqref{eq:RDE}.
They are geometric integrators
because they preserve some of the qualitative properties of the exact solution.

\subsection{Structure preserving integrators}
We are interested in numerical integrators which preserve the
symmetry as well as the positivity of $P(t)$. While symmetry is a
property preserved by most
methods, the preservation of positivity is a more challenging
task.

For our analysis, it is convenient to review some results from the
numerical integration of differential equations. 
Given the
ordinary differential equation (ODE)
\begin{equation}\label{eq:ODE-t-dep}
  \dot x=f(x,t), \qquad x(t_0)=x_0\in\mathbb{R}^n,
\end{equation}
the exact solution at time $t=t_0+h$ can formally be written as a
map that takes initial conditions to solutions,
$\Phi_h(x_0)=x(t_0+h)$. For sufficiently small $h$, it can also be
interpreted as the exact solution at time $t=t_0+h$ of an
autonomous ODE
\[
 \dot x=f_h(x), \qquad x(t_0)=x_0,
\]
where $f_h$ is the vector field associated to the Lie operator 
$\ %
\frac{1}{h}\log(\Phi_h)$.

In a similar way, a numerical integrator for solving the equation
(\ref{eq:ODE-t-dep}) which is used with a time step $h$, can be
seen as the exact solution at time $t=t_0+h$ of a perturbed
problem (backward error analysis)
\[
 \dot x=\tilde f_h(x), \qquad x(t_0)=x_0,
\]
and we say that the method is of order $p$ if $\tilde f_h- f_h=\mathcal{O}(h^{p+1})$. 
The qualitative properties of the exact solution
$\Phi_h$ are related to the algebraic structure of the vector
field $f_h$: If the vector field $\tilde f_h$ associated with
the numerical integrator shares the same algebraic structure, the
numerical integrator will preserve these qualitative properties.

Given the RDE
\begin{displaymath}
 \dot{P}=PA(t) +A^{T}(t)P -PS(t)P +Q(t), \qquad P(t_0)=0,
\end{displaymath}
which is equivalent to (\ref{eq:RDE}) with the sign of the time
changed, i.e., integrated backward in time, and with $Q(t),S(t)$
symmetric and positive definite matrices, then $P(t)$, for
$t>t_0$, is also a symmetric and positive definite matrix. Thus, a
numerical integrator which can be considered as the exact solution
of a perturbed matrix RDE
\begin{displaymath}
 \dot{P}=P \tilde A_h +\tilde A^{T}_hP -P\tilde S_hP +\tilde Q_h, \qquad P(t_0)=0,
\end{displaymath}
where $\tilde Q_h,\tilde S_h$ are symmetric positive definite
matrices will preserve the symmetry and positivity of the exact
solution.
The same result applies if the  numerical integrator is given by a
composition of maps such that each one, separately, can be seen as
the exact solution of a matrix RDE with the same structure.

We will refer to these methods as
\textit{positivity preserving integrators}. If a method preserves
positivity for all $h>0$, we say it is \textit{unconditionally
positivity preserving} and, if there exists $h^*>0$ such that this
property is preserved for $0<h<h^*$, we will refer to it as
\textit{conditionally positive preserving}.

In general, standard methods do not preserve positivity. We show,
however, that some second order exponential integrators preserve
positivity unconditionally and higher order ones are
conditionally positivity preserving for a relatively large range
of values of $h^*$ which depends on the smoothness in the time
dependence of the matrices $A(t),S(t),Q(t)$.
At this stage, it is convenient to rewrite the RDE (\ref{eq:RDE}) as
a linear differential equation
\begin{equation}\label{identification}
 \frac{d}{dt}\begin{bmatrix} V(t) \\ W(t) \end{bmatrix}
 = \begin{bmatrix}
      -A(t)^{T}  & - Q(t) \\
      - S(t)     &   A(t)
    \end{bmatrix}\,
  \begin{bmatrix} V(t) \\ W(t) \end{bmatrix}, \quad
    \begin{bmatrix} V(t_f) \\ W(t_f) \end{bmatrix} = \begin{bmatrix}P_f \\ I \end{bmatrix},
\end{equation}
where $P_f=0$, $S(t)=B(t) R^{-1} (t) B^{T} (t)$ and  the solution
$P(t)$ of problem \eqref{eq:RDE}, to be integrated backward in time,
is given by
\begin{displaymath}
        P(t)= V(t) W(t)^{-1}, \quad
        P(t), V(t) , W(t) \in
\mathbb{R}^{n \times n },
\end{displaymath}
in the region where $ W(t) $ is invertible (see, for instance, \cite{applnum} or
\cite{JP95}, and references therein). If $R(t)$ and $Q(t)$ are positive
definite matrices, this problem always has a solution.

It is then clear that if a numerical integrator for the equation
(\ref{identification}) can be regarded as the exact solution of an
autonomous perturbed linear equation
\begin{equation*}\label{identification2}
 \frac{d}{dt}\begin{bmatrix} V(t) \\ W(t) \end{bmatrix}
 = \begin{bmatrix}
      -\tilde A_h^{T}  & - \tilde Q_h \\
      - \tilde S_h     &   \tilde A_h
    \end{bmatrix}\,
  \begin{bmatrix} V(t) \\ W(t) \end{bmatrix}, \quad
    \begin{bmatrix} V(t_f) \\ W(t_f) \end{bmatrix}
    = \begin{bmatrix}P_f \\ I \end{bmatrix},
\end{equation*}
where $\tilde Q_h$ and $\tilde S_h$ are symmetric and positive
definite matrices, then the numerical solution is symmetric and
positive definite.

In general, high order standard methods like Runge-Kutta methods
do not preserve positivity. Explicit methods applied to the linear
problem do not preserve positivity unconditionally, but to show
this result for implicit methods requires a more detailed
analysis and it is stated in the following theorem.

\begin{theorem}
 The second order implicit midpoint and trapezoidal Runge-Kutta
methods do not preserve positivity unconditionally for the
solution of the RDE (\ref{identification}).
\end{theorem}

\begin{proof}
 It suffices to prove it for the scalar non-autonomous problem
\begin{equation}\label{eq:scalar}
  \dot p = -q-2a(t)\, p + s\, p^2, \qquad  \quad p(t_f)=0
\end{equation}
with $q,s>0$ and $a:[0,t_f]\to \mathbb{R}$.
Let $a(t)=0$ for $t\in[t_f-h,t_f]$ and $a(t)=a<0$ for $t\in[0,t_f-3h/2]$.
Then, for the implicit midpoint rule, two iterations backwards in time starting with $p_0=0$ yield a negative value $p_2<0$ if $|a|>2/h$.
For the trapezoidal rule, three iterations are necessary to reach negative values, with a sufficient condition for $p_3<0$ being
$$a< -\frac{3}{h}+h q \left(-\frac{1}{4}+\frac{6}{-4+h^2 q}\right) \quad\wedge\quad h<\frac{2}{\sqrt{q}}.$$

We remark, that, given $a<0$, the method produces negative values $p_3$ for a range of time-steps, i.e., for larger time-steps $h$, it is less prone to negativity.
Furthermore, not only can the methods produce negative values $p_k$, for certain parameter ranges they also attain complex valued results.
\end{proof}
\begin{figure}\centering
\includegraphics[width=\textwidth]{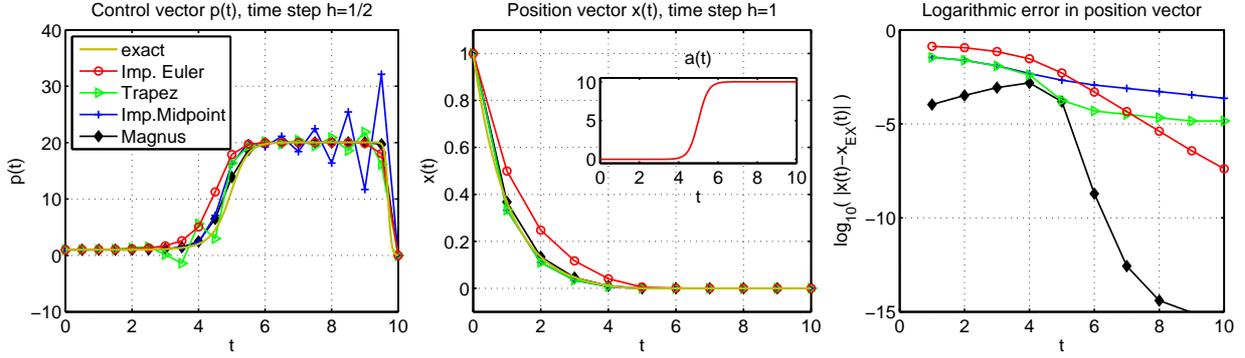}
\caption{\label{fig:comparison_scalar} (color online) Exact and
numerical solutions for the problem (\ref{eq:scalar}) with
parameters given in  (\ref{eq:scalarParameters}) and using the
implicit Euler, trapezoidal and midpoint Runge-Kutta methods as
well as the second order Magnus integrator.
}
\end{figure}

To better illustrate the results, we consider the problem
(\ref{eq:scalar}) for
\begin{equation}\label{eq:scalarParameters}
  q=s=1, \qquad a(t)=\frac{10}{1+\exp(-4(t-t_f/2))},
 \qquad t_f=10,
\end{equation}
with  $h=1/2$ for the equation of $p(t)$ and  $h=1$ for the
equation of $x(t)$,
$$
\dot x = \left(a\left(t\right)- s p\left(t\right)\right) x,\qquad  x(0)=0.
$$
We integrated the problem using the second
order implicit trapezoidal and midpoint methods as well as the
first order implicit Euler method. The results are shown in
Figure~\ref{fig:comparison_scalar}, where we appreciate that the
first order implicit Euler method is superior (the results obtained with the second order Magnus integrator to be presented in the next section is also included). 
The poor performance and non-positivity of the higher order standard
implicit methods is manifest.

If we are interested in high order numerical integrators,
different classes of methods have to be explored. We consider a
particular class of exponential integrators referred to as Magnus
integrators (see \cite{report} and references therein).

\subsection{Magnus integrators}\label{section_Magnus}

Given the general linear equation
\begin{equation}\label{eq:LinGeneral}
 y' = M(t) \, y \, , \qquad \quad y(t_0) = y_0 \,,
\end{equation}
with $y\in\mathbb{R}^p$, and if we denote the fundamental solution by $ \Phi(t,t_{0}) \in \mathbb{R}^{p \times p} $, such that
 $%
 y(t) = \Phi (t,t_{0}) y (t_{0}), %
 $
the Magnus expansion gives us the formal solution (under certain convergence conditions, see \cite{report} and references therein) as
\[
  \Phi (t,t_{0}) = \exp\left(\Omega(t,t_{0})  \right),
\]
where $\Omega(t,t_{0})=\sum_{n=1}^{\infty}\Omega_n(t,t_{0})$ and each $\Omega_n(t,t_{0})$ is an element of the Lie algebra generated by $n$-dimensional integrals involving $n-1$ nested commutators of $M(t)$ at different instants of time.
The first two terms are given by
\begin{displaymath}
  \Omega_1(t,t_{0}) = \int_{t_0}^{t} M(s) \, ds, \qquad
  \Omega_2(t,t_{0}) = \frac12 \int_{t_0}^{t}  \, dt_1 \int_{t_0}^{t_1}
  [M(t_1),M(t_2)] \, dt_2,
\end{displaymath}
where $[A,B]=AB-BA$.

In the region of convergence of the Magnus expansion, the exact solution at time $t=t_0+h$ is equivalent to the exact solution of the autonomous linear equation
\begin{displaymath}
 y' = \frac{1}{h} \Omega(t_0+h,t_{0}) \, y \, ,
 \qquad \quad y(t_0) = y_0 \, .
\end{displaymath}
It is well known that the set of matrices
\begin{equation}\label{eq:Symplectic}
 \left[ \begin{array}{cc}
  A & B \\ C & -A^T
 \end{array} \right],
\end{equation}
with $A,B,C\in\mathbb{R}^{n\times n}$ and $B=B^T, \ C=C^T$ form
the algebra of symplectic matrices.
This algebraic property is preserved by the commutators and then any truncated Magnus expansion preserves symplecticity for this problem.
However, the additional property about the positivity (or negativity) of the skew diagonal matrices $B,C$ is not always guaranteed when the series is truncated. We analyze lower order methods and show that it is possible to build second order Magnus integrators which unconditionally preserve positivity.

The first term in the expansion applied to \eqref{identification} does not contain commutators and is given by
\begin{equation}\label{eq:Magnus1}
  \Omega_1(t,t_{0}) %
   = \begin{bmatrix}
\displaystyle      -\int_{t_0}^{t}A(s)^{T} \, ds  & \displaystyle - \int_{t_0}^{t}Q(s) \, ds \\ & \\
\displaystyle      - \int_{t_0}^{t}S(s) \, ds  &   \displaystyle \int_{t_0}^{t}A(s) \, ds
    \end{bmatrix}\,.
\end{equation}

Then, if we truncate the series after the first term and approximate the integrals for a time interval $t\in[t_0,t_0+h]$ using a quadrature rule of second or higher order, we obtain a second order method.

It is well known that, if $Q(t)$ is a symmetric positive definite
matrix for $t\in[t_0,t_0+h]$, then
$ \ %
\hat Q_h=\int_{t_0}^{t_0+h}Q(s) \, ds \ $
is also symmetric  positive definite.
Suppose now that the integral is approximated using a quadrature rule
\[
 \tilde Q_h\equiv h \sum_{i=1}^k b_i Q(t_0+c_ih) \simeq
 \int_{t_0}^{t_0+h}Q(s) \, ds,
\]
with $\ c_i\in[0,1], \ i=1,\ldots,k$. If $ \ \sum_ib_i>0 \, $, we have:
\begin{description}
  \item[a)] If $b_i>0, \ i=1,\ldots,k$, then $\tilde Q_h$ is a
  symmetric positive definite matrix.
  \item[b)] If  $\exists \, b_j<0, \ $ for some value of $j$ and
  $\|Q(t_m)-Q(t_n)\|<C|t_m-t_n|, \ \forall t_m,t_n\in[t_0,t_0+h]$,
  then $\exists \ h^*>0$ such that $\tilde Q_h$ is a
  symmetric positive definite matrix for $0<h<h^*$, and $h^*$ depends on $C$.
\end{description}
The same results also apply to $\tilde S_h$.

A second order method which preserves positivity is constructed by taking the first term in the Magnus expansion (\ref{eq:Magnus1}) and approximating the integrals by a second or higher order rule with the constraint that all $b_i>0$. 
The most natural choices are the midpoint rule
\begin{displaymath}
  \Psi_h^{[2]} =  \exp \left( h  M(t+h/2) \right)
    = \Phi(t+h,t) + \mathcal{O}(h^3),
\end{displaymath}
or the trapezoidal rule
\begin{equation}\label{Trapezoidal}
  \Psi_h^{[2]} =  \exp \left( \frac{h}{2} \left[ M(t+h) + M(t) \right] \right)
    = \Phi(t+h,t) + \mathcal{O}(h^3).
\end{equation}
The latter of which is found more efficient since less evaluations
of the functions in the algorithm are necessary as they can be reused in the computation of $X(t)$. If we consider
the RDE \eqref{eq:RDE} that corresponds to (\ref{eq:LinGeneral})
with the data (\ref{identification}) and choose an equidistant
time grid $t_{n} = t_{0} + nh \, $, $ \ 0 \leq n \leq N \, $, with
constant time step $ \, h=(t_{f}-t_{0})/N \, $
and taking into account that this equation has to
be solved backward in time, we obtain
\begin{displaymath}
 \left[ \begin{array}{c} V_{n} \\ W_{n}
 \end{array} \right] = \exp\left( - \frac{h}{2} \left[ M(t_{n}) + M(t_{n+1}) \right]  \right)
 \left[ \begin{array}{c} V_{n+1} \\ W_{n+1}
 \end{array} \right] \ \Rightarrow \
 \tilde P_{n}= V_{n}\,W_{n}^{-1} \, ,
\end{displaymath}
By construction, $\tilde P_{n}$ is a symmetric positive definite
matrix. In addition, it is also a time symmetric second order
approximation to $P(t_{n})$.
In this way, the matrix functions $A(t_{n})$, $B(t_{n})$,
$Q(t_{n})$, $R(t_{n})$ are computed at the same mesh points as the
approximations $\tilde P_h$ of $P(t)$ and, as we will
see, they can be reused for the forward time integration of the
state vector.

Let us consider the equation for the state vector, to be integrated forward in time, which takes the form
\begin{displaymath}
  \dot X = (A(t)-S(t)\tilde P_h) X,
\end{displaymath}
where we denote by $\tilde P_h$ the numerical approximations to $P(t)$ computed on the mesh and  $\tilde P_{h,n}\simeq P(t_n)$.
Notice that at the instant $t=t_f-h$, we have that $\tilde P_{h,N-1}=P(t_f-h)+\mathcal{O}(h^3)$ (local error) but at $t=t_0$, after $N$ steps, we have $\tilde P_{h,0}=P(t_0)+\mathcal{O}(h^2)$ (global error).
This accuracy suffices to get a second order approximation for the numerical approximation to the state vector.

If we use the same Magnus expansion for the numerical integration of the state vector with the trapezoidal rule, we have the algorithm
\begin{equation*}\label{Xn}
        X_{n+1} = \exp \left(
            \frac{h}{2} \left[
        D_{n+1} + D_{n}
         \right] \right) \, X_{n} \, ,
  \ \quad     D_{m} = A_{m} - S_{m} \tilde P_{h,m},
   \quad  m=n,n+1,
\end{equation*}
where $A_{m}=A(t_{m})$, $S_{m}=S(t_{m})$.

Finally, the controls which allow us to reach the final state in a
nearly optimal way are computed via
 \[
 u_n \,  = \, - R^{-1} (t_n) B^{T} (t_n)  P_n \, X_n \, .
 \]

\paragraph{Higher order Magnus integrators}

Truncating the Magnus expansion at higher powers of $h$ usually requires the computation of matrix commutators.
If we include, for example, the second term $\Omega_2$ in the exponent, we obtain $\Psi_h\equiv\exp\left(\Omega_1+\Omega_2\right)$, which agrees with the exact solution up to order four, i.e., $\Psi_h=\Phi (t+h,t)+ \mathcal{O}(h^5)$.
The sum $\Omega_1+\Omega_2$ belongs to the algebra of symplectic
matrices, as given in (\ref{eq:Symplectic}), where the
off-diagonal matrices $B,C$
take an involved form.
We will show that positivity is conditionally preserved, however, unconditional preservation as for $\exp\left(\Omega_1\right)$ cannot be achieved.

For simplicity in the analysis, we consider commutator-free Magnus
integrators (see \cite{report,blanes06fas} and references
therein). With the abbreviations
\begin{displaymath}
  M^{(0)}= \int_{t_n}^{t_n+h} M(s) \, ds, \qquad
  M^{(1)}= \frac{1}{h} \int_{t_n}^{t_n+h}
  \left(s-\left(t_n+\frac{h}{2}\right)\right) M(s) \, ds,
\end{displaymath}
the following commutator-free composition gives an approximation to fourth-order
\begin{displaymath}
 \Psi_{CF}^{[4]} = \exp\left(\frac12 M^{(0)}+2M^{(1)}  \right) \,
  \exp\left(\frac12 M^{(0)}-2M^{(1)}  \right) =
  \Phi(t_0+h,t_0) + \mathcal{O}(h^5).
\end{displaymath}
Using the fourth-order Gaussian quadrature rule to approximate the integrals yields
\begin{displaymath}
 \Psi_{G}^{[4]} = \exp \left(  h (\beta M_1 + \alpha M_2) \right)
 \exp \left(  h (\alpha M_1 + \beta M_2) \right),
\end{displaymath}
where $M_i\equiv M(t_n+c_ih), \ i=1,2$,
$c_1=\frac12-\frac{\sqrt{3}}{6}, \
c_2=\frac12+\frac{\sqrt{3}}{6}$,
$\alpha=\frac14-\frac{\sqrt{3}}{6}=-0.038\ldots<0, \
\beta=\frac14+\frac{\sqrt{3}}{6}$. This composition will not
preserve positivity unconditionally when applied to solve the RDE
because $\alpha<0$. However, since $\alpha+\beta=\frac12$ the
positivity will be conditionally preserved.

If we approximate the integrals using the Simpson rule, we have
\begin{displaymath}
 \Psi_{G}^{[4]} = \exp \left(  \frac{h}{12} (-M_1 + 4M_2 + 3M_3) \right)
 \exp \left(  \frac{h}{12} (3M_1 + 4M_2 - M_3) \right),
\end{displaymath}
where $M_1\equiv M(t_n), M_2\equiv M(t_n+h/2), M_3\equiv M(t_n+h)$.
As previously, one of the coefficients is negative and positivity is not unconditionally preserved when the method is applied to the RDE.

Recall that the full problem requires the solution of two differential equations; suppose we want to (backward) integrate the matrix RDE with one of the fourth-order commutator-free methods and then use the same method for the (forward) integration of the state vector, we need to use a time step twice as large for the forward integration (preferably with the Simpson rule, since no interpolation is necessary).

The main goal of this paper is to present a simple, fast, accurate and reliable numerical scheme for nonlinear problems.
As we will see, nonlinear problems are linearized, and the resulting linear equations are solved iteratively.
The solution of each iteration is plugged into the following iteration, and this requires to use a fixed mesh for all methods.
For this reason, the second order Magnus integrator is the optimal candidate among the previous and is used in the numerical examples in section~\ref{quadrotor}.

\section{The nonlinear control problem}\label{nonlinear}

Many problems in engineering can be stated as optimal control problems of the form
\begin{subequations}\label{eq:nonlin}
\begin{align}\label{eq:nonlin:cost}
    \underset{u\in L^2}{\min} \int_0^{t_f}
            \left( X^{T}(t)Q(t,X(t)) X(t) + u^{T}(t) R(t,X(t))u(t) \right) \, dt
    \\ \label{eq:nonlin:x}
    \text{subject to} \quad
            \dot X(t) = f_A\left(t,X\left(t\right)\right)
                                + f_B\left(t,X\left(t\right),u\left(t\right)\right),
                                 \quad X(0) = X_0.
\end{align}
\end{subequations}

This nonlinear optimal control problem is considerably more involved than its linear counterpart.
It is then usual to solve the nonlinear problem by linearization, and this can be done in different ways.
In the following, under the assumption that $f_B$ depends linearly on $u$, we present three of them and compare their performances when the linear equations are solved using exponential integrators.

\paragraph{Quasilinearization}
For $f_A(t,0)=0$ and $f_B(t,X,u)\neq 0$ for all $t,X$ in the
appropriate domains, the state equation \eqref{eq:nonlin:x} can be
written in a non-unique way as 
\beq\label{eq:nonlin:lin}
    \dot X(t) = A(t,X)X(t) + B(t,X)u(t), \quad X(0)=X_0.
\eeq

The formulation \eqref{eq:nonlin:lin} is the basic ingredient for
the State Dependent Riccati Equation (SDRE) control technique
\cite{cimen08sdr,Cloutier97}. Its formal similarity to the
linear problem \eqref{eq:lqproblem} motivates the imitation of the
optimal LQ controller by defining
\begin{subequations}\label{eq:sdre}
\beq\label{eq:sdre:u}
    u(t) = - R^{-1}(t) B^T(t,X(t))P(t,X) X(t)
\eeq where $P(t,X)$ solves the now state-dependent algebraic
Riccati equation \beq\label{eq:sdre:are}
    0 = - P A(t,X) - A(t,X)^TP + P B(t,X)R(t,X)^{-1}B(t,X)^T P - Q(t,X).
\eeq 
One has to choose the unique positive definite solution of
the algebraic Riccati equation and, combining \eqref{eq:sdre:u}
with \eqref{eq:nonlin:lin}, the closed-loop nonlinear dynamics are
given by \beq\label{eq:x:sdre:nonlin}
  \dot X = \left( A(t,X)-B(t,X)R(t,X)^{-1}B(t,X)^TP(t,X) \right) X,
                \quad X(0)=X_0.
\eeq
\end{subequations}
The usual approach is to start from $X(0)=X_0$, and then to
advance step by step in time by first computing $P$ from
\eqref{eq:sdre:are} at each step and then applying the Forward
Euler method on \eqref{eq:x:sdre:nonlin}. The application of
higher order methods, such as Runge-Kutta schemes, requires to
solve implicit systems with \eqref{eq:sdre:are} and can thus be
costly. In addition, if one is interested in aggressive
trajectories, the algebraic equation (\ref{eq:sdre:are}) can
considerably differ from the solution of the corresponding Riccati
differential equation, which affects the solution of the state
vector, $X$, and ultimately the control in \eqref{eq:sdre:u}.

\paragraph{Waveform relaxation}
Alternatively, we can linearize \eqref{eq:x:sdre:nonlin}, by
iterating
 \beq\label{eq:x:mixed:iteration}
    \frac{d}{dt} X^{n+1} = \left(A(t,X^n)-B(t,X^n)R(t,X^n)^{-1}B(t,X^n)^TP(t,X^n)
                                \right) X^{n+1}.
\eeq
We start with a guess solution $X^0(t)$ and iteratively obtain a
sequence of solutions, $X^1(t)$, $X^2(t),\ldots,X^n(t)$. 
Again, the iteration stops once
consecutive solutions differ by less than a given tolerance.
Here, $P(t,X^n(t))$ at each iteration is obtained from
\beq\label{eq:rde}
    \dot P = - PA^n(t) - A^n(t)^TP + P B^n(t)R^n(t)^{-1}B^n(t)^TP -
    Q^n(t),\quad P(t_f)=0,
\eeq with $A^n(t) \equiv A(t,X^n(t))$, $B^n(t)\equiv B(t,X^n(t))$,
etc.

This procedure is similar to what is known as waveform relaxation
\cite{White85}, however, the backward integration for $P$ limits
the parallelizability in this application. This approach
corresponds to freezing the nonlinear parts in
\eqref{eq:nonlin:lin} at the previous state and then applying the
optimal control law \eqref{eq:linear:opt_control}. It is worth
noting that this technique can handle inhomogeneities by slightly
adapting the control law, at the cost of solving an inhomogeneous
linear system, see below. The algorithms are illustrated in
Table~\ref{algorithm:linearization}.

\begin{table}[h!]
\centering
\hspace{-2mm}
 {\small\begin{tabular}{l|l}
  \hline
  {\bf A1: waveform relaxation}
  &
  {\bf A2: linearization}\\
  \hline
  \rule[1mm]{0mm}{2cm}$%
  \begin{array}{l}
        n := 0; \quad \textit{guess}: X^0(t), u^0(t)\\
              {\bf do} \\
                \quad n := n+1\\
                    \quad \text{compute}:%
                                    A^{n-1}(t), B^{n-1}(t)\\
                  \quad \text{solve}\, (t_f\to0):
                        \text{eq. }\eqref{eq:rde} \textit{ for } P^{n-1}\\
                    \quad \text{solve}\, (0\to t_f):
                            \text{eq. }\eqref{eq:x:mixed:iteration}
                                                                \textit{ for } X^{n}\\
        {\bf while}\ |X^{n}-X^{n-1}|>tolerance  \\[1mm]
        \textit{Check for feasibility of }X^{n}\\
                \\
   \end{array}
   $
   &
   $
   \begin{array}{l}
        n := 0; \quad \textit{guess}: X^0(t), u^0(t)\\
          {\bf do} \\
                \quad n := n+1 \\
                                    \quad \text{compute}:%
                                    \bar A^{n-1}(t), \bar B^{n-1}(t), \bar C^{n-1}(t)\\
                  \quad \text{solve}\, (t_f\to0):
                        \text{eq. }\eqref{eq:RDE} \textit{ for } P^{n-1}\\
                  \quad \text{solve}\, (t_f\to0):
                        \text{eq. }\eqref{eq:disturbance} \textit{ for } V^{n}\\
                    \quad \text{solve}\, (0\to t_f):
                            \text{eq. }\eqref{eq:linearized:inhomogeneous}
                                                                \textit{ for } X^{n}\\
        {\bf while} \  |X^{n}-X^{n-1}|>tolerance \\[1mm]
        \textit{Check for feasibility of }X^{n}\\
      \end{array}
    $
    \\ \hline
 \end{tabular}
}
  \caption{Algorithm (A1) for the waveform relaxation
         and algorithm (A2) for the Taylor-type linearization.
         \label{algorithm:linearization}
         }
\end{table}

\paragraph{Taylor-type linearization}
Similarly to \cite{ponsoda11nen}, 
we can Taylor-expand the vector field in \eqref{eq:nonlin:x} around an approximate solution $X^n(t)$ and use optimal LQ controls for the approximated equation.
The iteration step reads then
\beq\label{eq:linearized:inhomogeneous}
    \dot X^{n+1}(t) =    \bar A^n(t) X^{n+1}(t)
                                        +  \bar B^n(t) u^{n+1}(t) + \bar C^n(t),
\eeq where
\begin{displaymath}
\begin{array}{rl}
    \bar A^n(t)&=D_Xf_A\left(t,X^n(t)\right) +
                                         D_Xf_B\left(t,X^n(t),u^n(t)\right),\\
    \bar B^n(t)&=D_Uf_B(t,X^n,U^n),\\
    \bar C^n(t) &=
    f_A(t,X^n)+f_B(t,X^n,u^n) -
    \left(
                \bar A^n(t)
                \cdot X^n +
                                 \bar B^n(t)
                \cdot u^n  \right),
                \end{array}
\end{displaymath}
and $D_X$ denotes the derivative with respect to $X$, etc. One
starts with an initial guess, $X^0(t)$ and the iteration stops
once consecutive iterations differ by less than a given tolerance.

The inhomogeneity $\bar C^n$ can be treated as a disturbance input
and compensated by the controller \cite{Bryson75}. The optimal
control then becomes
\begin{displaymath}
    u^{n+1}(t) = - {R^n(t)}^{-1} {\bar B^n(t)}^{T}
                                    \left(P^n(t) X^{n+1}(t)  + V^{n}(t)\right)
\end{displaymath}
where $P^n(t)$ satisfies \eqref{eq:RDE} with replacements $A\to
\bar A^n$ and $B\to \bar B^n$, etc. and $V^n(t)$ is given by
\begin{align}
    \label{eq:disturbance}
    \dot V & = \left( P\bar B R^{-1} {\bar B}^{T} - {\bar A}^{T} \right) V
                                    - P\bar C, \quad V(t_f)=0
\end{align}
at each iteration.
The linearization procedure is summarized in Table
\ref{algorithm:linearization}.

\vspace*{.4cm}

\noindent NOTE: We can solve non-homogeneous equations with Magnus
integrators as follows. Given the non-homogeneous equation
\begin{displaymath}
 y' = M(t) \, y \, + C(t) , \qquad \quad y(t_0) = y_0 \, ;
\end{displaymath}
it can be formulated as a homogeneous one in the following way
\cite{applnum},
\begin{displaymath}
 \frac{d}{dt} \begin{bmatrix} y \\ 1 \end{bmatrix}
    \ = \ \begin{bmatrix} M(t) & C(t) \\
                    0_n^T        & 0
         \end{bmatrix}\,
\begin{bmatrix} y \\ 1 \end{bmatrix}, \quad
[y(0),1]^T=[y_0,1]^T,
\end{displaymath}
where $0_n^T=[0,\ldots,0]\in\mathbb{R}^n$.

\section{Modeling the control of a quadrotor UAV
}\label{quadrotor}

The optimal control of Unmanned Aireal Vehicles (UAV) has attracted great attention in recent years \cite{budiyono07otc,castillo}.
Helicopters are classified as Vertical Take Off Landing (VTOL) aircraft and are among the most complex flying objects because their flight dynamics are nonlinear and their variables are strongly coupled.

In this section, we address the optimal control of a quadrotor, i.e., a vehicle with four propellers, whose rotational speeds are
independent, placed around a main body \cite{tesis,ieee,castillo04rts,castillo,cow}.
Linear techniques to control the system have been frequently used.
The controllers are designed based on a simplified description of the system behavior (linearized models).
While this is satisfactory at hover and low velocities, it does not predict correctly the system behavior during fast maneuvers (most
controllers are specifically designed for low velocities) and in order to improve the performance, the nonlinear nature of the quadrotor has to be taken into account \cite{cimen08sdr,voos06nsd}.
In addition, problems can have time-varying parameters \cite{zhang10ato} or require time-dependent state references \cite{cow}.

Under realistic conditions, real time calculations are necessary since the optimal control will have to adjust to environmental changes, that are not accounted for in the model, and hence more efficient and elaborated algorithms have to be designed.

LQ optimal controllers are widely used, in particular
for the control of small aircrafts \cite{tesis, voos06nsd},
where they have shown to produce better results
than other standard methods, like
proportional integral derivative methods (PID) \cite{ieee}.
The techniques presented here, however, are valid for the general
optimal LQ control problem (\ref{eq:lqproblem}).

For the illustration of our methods, we consider a VTOL quadrotor,
based on the model presented in \cite{castillo,voos06nsd} (and
references therein). Figure \ref{fig:quadrotor} describes the
configuration of the system, where $ \phi $, $ \theta $ and $ \psi
$ denote the rolling, pitching and yawing angles, respectively.
\begin{figure}[!ht]
    \begin{center}
      \input{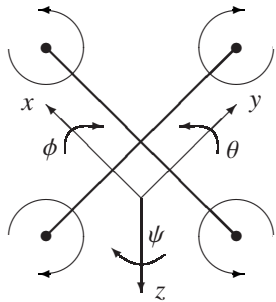}
    \end{center}
    \caption{{\small{Quadrotor schematic}}}
    \label{fig:quadrotor}
\end{figure}
We assume some standard general conditions on the symmetric and rigid structure of the flying robot: the center of mass is in the center of the planar quadrotor and the propellers are rigid.

We remark that inhomogeneities
$f_A(t,0)=b(t)$, e.g., from gravitational forces, can be treated as
disturbances, by adding new state variables or by taking advantage
of non-vanishing states, e.g., the altitude of the UAV when hover
is searched \cite{cimen08sdr}. 

An analysis of the dynamics of the quadrotor shows that the
control of the attitude can be separated from the translation of
the UAV \cite{voos06nsd} and we focus our attention on the
stabilization of the attitude, neglecting the gyroscopic effect.
The state vector is given by
\begin{displaymath}
X(t)  =  \left(\phi(t), \dot\phi(t), \theta(t), \dot\theta(t), \psi(t), \dot\psi(t)\right)^T \in \mathbb{R}^{6}  ,
\end{displaymath}
and the input vector $u \in \mathbb{R}^{3}$
is formed by linear combinations of the thrust of each propeller.

The system designer can choose the weight matrices to tune the
behavior of the control according to the requirements, $R(t)$ is
used to suppress certain movements and $Q(t)$ limits the use of
the control inputs. Usually, these matrices are chosen constant,
positive definite and often even diagonal, see \cite[p.
67]{tesis}, \cite{castillo,cow}. For the numerical experiments, we
have implemented the problem (\ref{eq:nonlin}) with the following
values taken from \cite{BouabdallahTesis,voos06nsd}
\beq\begin{array}{lll}\label{eq:def:AyB}
a_{1,2}=a_{3,4}=a_{5,6}=1,& a_{2,4}=\lambda\alpha_1 I_1 \dot\psi,&
a_{2,6}=\lambda(1-\alpha_1) I_1 \dot\theta
\\
a_{4,2}=\lambda\alpha_2 I_2 \dot\psi,&
a_{4,6}=\lambda(1-\alpha_2) I_2 \dot\phi,&
a_{6,2}=\lambda\alpha_3 I_3 \dot\theta,\\
a_{6,4}=\lambda(1-\alpha_3)I_3 \dot\phi,& b_{2,1} = L/I_x, \ \
b_{4,2} = L/I_y,& b_{6,3} = 1/I_z
\end{array}
\eeq where $\alpha_i$ reflects the non-uniqueness in the SDRE
formulation, $\lambda$ denotes the inflow ratio, $L$ is the length of the arms connecting the
propellers with the center and the relative moments of inertia are $I_1=(I_y-I_z)/I_x$,
$I_2=(I_z-I_x)/I_y$, $I_3=(I_x-I_y)/I_z$. Here, $\, m_{i,j} \, $
denotes the element located at $i$-th row and $j$-th column of the
matrix $M$. Other entries of $A(t) \in \mathbb{R}^{6 \times 6}$
and $B(t) \in \mathbb{R}^{6 \times 3}$ not indicated in
\eqref{eq:def:AyB} are null elements.

The numerical values are extracted from \cite{BouabdallahTesis} and
are given in the SI units
\begin{displaymath}
I_{x} = 0.0075, \quad I_{y} = 0.0075,  \quad  I_{z} = 0.0130,
\quad  L = 0.23, \quad  \lambda = 1, \quad \alpha_{i}=1.
\end{displaymath}
The weight matrices are fixed at
\begin{displaymath}
Q = 0.01\cdot\text{diag} \{ \, 1, 0.1 , 1, 0.1, 1, 0.1\}\in \mathbb{R}^{6 \times 6},\quad
R = \text{diag} \{1, 0.1, 1 \} \in \mathbb{R}^{3 \times 3}.
\end{displaymath}
We set the time frame to $t_{f}=10$ seconds, with a stepsize of $h=0.125s$
and initial state
\begin{displaymath}
X_{0}  =  \left(70\degree,\, 10,\, 70\degree,\, 20,\, -130\degree\,, -1\right)^{T},
\end{displaymath}
that corresponds to a disadvantageous orientation and high
rotational velocities that are sought to be stabilized at
$0\in\R^6$ at the final time $t_{f}$.

We have implemented a variety of methods  to test against the
Magnus integrators presented in section \ref{section_Magnus}.
As initial condition, we have taken $X^0(t) = (1-t/t_f)X_0$ and
the iteration was stopped when $||X^n-X^{n-1}||_2<10^{-3}$. We use
the explicit and implicit Euler methods as well as the second
order Magnus integrator. Some experimental results are given in
Table~\ref{table:results}, where we can see that the Magnus based
method \eqref{Trapezoidal}, approximates the optimal control best.
However, we have to remark that the SDRE method is for the given
parameters about a factor ten faster, due to necessary iterations
for the other schemes.

\begin{table}[h]\centering
    \begin{tabular}{lllllcr}
        &Type&  $X(t)$ & $P(t)$ & $V(t)$ & Cost  & It.\\ \hline\hline
        \texttt{S1})&SDRE & Euler                               & are                                   & N/A   & 0.1114\\
        \texttt{S2})&           & Impl. Euler (IE)          & are                                   & N/A       & 0.1021\\
                          &         & Optimal                           &$\Rightarrow$                  &               & 0.0977\\ \hline
        \texttt{W1})&WAVE & Euler                               & Euler                                 & N/A   & 0.1071& 3\\
        \texttt{W2})&           & IE                                  & IE                                      & N/A       & 0.1036& 3\\
        \texttt{W3})&         & Magnus \eqref{Trapezoidal} & Magnus                              & N/A   & 0.0926& 3\\
                         &        & Optimal                       &$\Rightarrow$                & N/A   & 0.0888\\ \hline
     \texttt{T1})&TAYLOR& Euler                                 & Euler                                 & Euler & N/A    & Inf\\
     \texttt{T2})&          & IE                                    & IE                                        & IE        & 0.0789 & 12\\
     \texttt{T3})&          & Magnus                                & Magnus                                & Magnus& 0.0707 & 12\\
                         &          & Optimal                               &$\Rightarrow$                  &               & 0.0707\\ \hline
    \end{tabular}
    \caption{Comparison of numerical methods, Type indicates the linearization procedure given by section \ref{nonlinear} and It. denotes the number of iterations necessary until convergence. The cost is a discrete approximation of the integral in \eqref{eq:nonlin:cost}.
                    \label{table:results}}
\end{table}

Figure~\ref{figure1} shows the controls obtained for the schemes
\texttt{S2}, \texttt{W3}, \texttt{T3} and Figure~\ref{figure2}
shows the motion of the quadrotor angles subject to the controls. 
We can appreciate how the Magnus methods maximize the
use of the controls to reach an overall minimum of the cost
functional.

\begin{figure}[h!]
\begin{center}
\psfrag{S2}[][l]{\;\;\;\;\texttt{S2}}
\psfrag{W3}[][l]{\;\;\;\;\texttt{W3}}
\psfrag{T3}[][l]{\;\;\;\;\texttt{T3}}
\includegraphics[width=\textwidth]{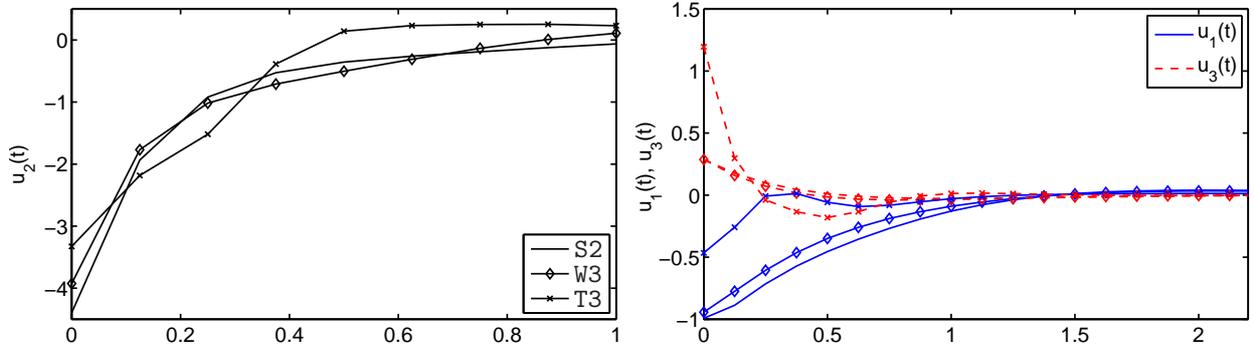}
\end{center}
\caption{\small{Evolution of the control vector. The left column
shows the control that has been least penalized $u_2$. All curves
are given for all methods \texttt{S2} (line), \texttt{W3}
(diamond) and \texttt{T3} (cross).
 \label{figure1}}}
\end{figure}

\begin{figure}[h!]
\begin{center}
\psfrag{S2}[][l]{\;\;\;\;\texttt{S2}}
\psfrag{W3}[][l]{\;\;\;\;\texttt{W3}}
\psfrag{T3}[][l]{\;\;\;\;\texttt{T3}}
\includegraphics[width=\textwidth]{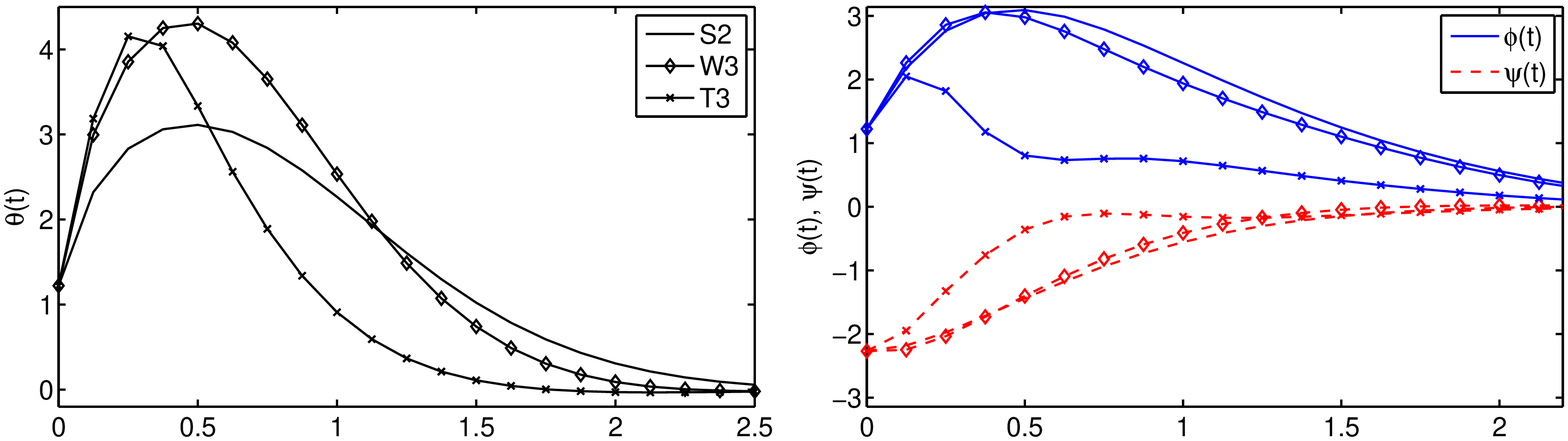}
\includegraphics[width=\textwidth]{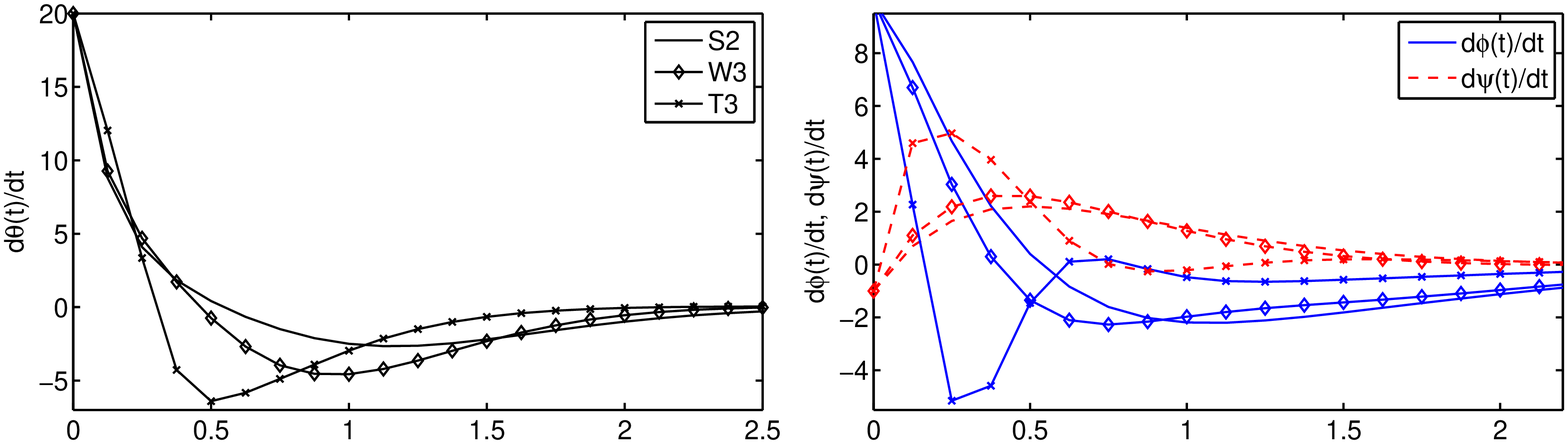}
\end{center}
\caption{\small{Evolution of the orientation of the quadrotor (top
row) and angular velocities (bottom). The left column shows the coordinates $\theta(t)$ and $\dot\theta(t)$, whereas the remaining coordinates $\phi(t), \dot\phi(t)$ and $\psi(t), \dot\psi(t)$ are depicted in the right column. All curves are
given for all methods \texttt{S2} (line), \texttt{W3} (diamond) and \texttt{T3} (cross).
 \label{figure2}}}
\end{figure}

From the numerical experiments we conclude that Lie group methods
such as Magnus integrators which preserve the positivity of the
solution of the matrix RDE are very useful tools for solving
optimal control problems of UAV.

\section{Conclusions}

We have presented structure preserving integrators based on the
Magnus expansion for solving linear quadratic optimal control
problems. The schemes considered require the numerical integration
of matrix RDEs whose solutions, for this class of problems, are
symmetric and positive definite matrices. The preservation of this
property is very important to obtain reliable and efficient
numerical integrators. While geometric integrators preserve most
of the qualitative properties of the exact solution, the
preservation of positivity for the matrix RDE is, in general, not
guaranteed. We have shown that some symmetric second order
exponential integrators (Magnus integrators) preserve this
property unconditionally and, in addition, are very appropriate to
build simple and efficient numerical algorithms for solving
nonlinear problems by linearization. The performance of the
methods is illustrated with an application to stabilize a quadrotor UAV. The results shown for a quadrotor easily extend to other helicopters.

For more involved trajectories, the structure of the equations
will play a more important role and the methods presented in this
work could be very useful in those cases. Additionally, in more
difficult settings, e.g., in the case of trajectory following or
obstacle avoidance, stronger time dependencies of the parameters
are expected, 
making standard methods more susceptible to instabilities, and thus, the advantages of the exponential methods are expected to be amplified. This tendency highlights these applications as interesting for further investigation.

\section*{Acknowledgments}

This work has been partially supported by Ministerio de Ciencia e
Innovaci\'on (Spain) under the coordinated project
MTM2010-18246-C03
(co-financed by the ERDF of
the European Union) and  MTM2009-08587, and the Universitat Polit\`{e}cnica de Val\`{e}ncia throughout
the project 2087.  P. Bader also acknowledges the support through the FPU
fellowship AP2009-1892.

\end{document}